\documentclass[12pt]{amsart}
\usepackage{mathrsfs}

\usepackage{amssymb,amsmath,graphicx,color,textcomp, amsthm,bbm,bbold, enumerate,booktabs}
\usepackage{chngcntr}
\usepackage{apptools}
\AtAppendix{\counterwithin{theorem}{section}}
\definecolor{Red}{cmyk}{0,1,1,0}

\definecolor{verde}{cmyk}{1,0,1,0}

\definecolor{loka}{cmyk}{.5,0,1,.5}
\definecolor{azul}{cmyk}{1,1,0,0}

\numberwithin{equation}{section}

\newcommand{\be}{\begin{equation}}
\newcommand{\ee}{\end{equation}}

\newtheorem{theorem}{Theorem}

\newtheorem{definition}{Definition}

\begin{document}
\title{Ulam-Hyers stability of a nonlinear fractional Volterra integro-differential equation}
\author{J. Vanterler da C. Sousa$^1$}
\address{$^1$ Department of Applied Mathematics, Institute of Mathematics,
 Statistics and Scientific Computation, University of Campinas --
UNICAMP, rua S\'ergio Buarque de Holanda 651,
13083--859, Campinas SP, Brazil\newline
e-mail: {\itshape \texttt{ra160908@ime.unicamp.br, capelas@ime.unicamp.br }}}

\author{E. Capelas de Oliveira$^1$}

\begin{abstract} Using the $\psi-$Hilfer fractional derivative, we present a study of the Hyers-Ulam-Rassias stability and the Hyers-Ulam stability of the fractional Volterra integral-differential equation by means of fixed-point method.

\end{abstract}
\maketitle

%%%%%%%%%%%%%%%%%%%%%%%%%%%%%%%%%%%%%%%%%%%%%%%%%%%%%%%%%%%%%%%%%%%%%%%%%%%%%%%%%%%%%%%%%%%%%%%%%%%%%%%%%%%%%%%%%%%%%%%%%%%%%%%%%%%%%%%%%%%%%%%%%%%%%%%%%%%%%

\section{Introduction}
The study of the fractional derivatives in systems of differential equations, has been of great applicability in mathematical models for example, involving problems of population dynamics, erythrocytes sedimentation rate, among others \cite{Ze,RHM}. On the other hand, studying the stability of Hyers-Ulam and Hyers-Ulam-Rassias, for solutions of fractional differential equations has been of great interest \cite{wang,sousa,JinRong,wei}. In addition, propose more general results that involve fractional integro-differential, has also been gaining prominence \cite{wang1,muniya}.

By proposing the study of solution stability via fractional integrals and fractional derivatives, we can generalize the results and obtain the usual ones as particular cases. However, for the variety of derivative and fractional integrals formulations, in this paper we will use two recent fractional operators, that is, of general differentiation and integration \cite{Ze1}.

The motivation for the elaboration of this paper is the study of the stability of Hyers-Ulam-Rassias and Hyers-Ulam of the following fractional nonlinear Volterra integro-differential equation
\begin{equation}  \label{eq1}
\left\{ 
\begin{array}{rcl}
^{H}\mathbb{D}_{0+}^{\alpha ,\beta ;\psi }u\left( t\right) & = & f\left(
t,u\left( t\right) \right) +\displaystyle\int_{0}^{t}k\left( t,s,u\left(
t\right) \right) ds \\ 
I_{0+}^{1-\gamma }u\left( 0\right) & = & \sigma%
\end{array}
\right.
\end{equation}
with $t\in I=[0,T]$, where $f\left( t,u\right)$ is a continuous function with respect to the variables $t$ and $u$ on $I\times \mathbb{R}$, $k\left( t,s,u\right) $ is continuous with respect to $t,$ $s$ and $u$ on $I\times \mathbb{R}\times \mathbb{R}$, $\sigma $ is a given constant, $^{H}\mathbb{D}_{a+}^{\alpha ,\beta ;\psi }\left( \cdot \right) $ (Eq.(\ref{eq3}), below) where $0<\alpha<1$, $0\leq \beta \leq 1$ and $I_{0+}^{1-\gamma }\left( \cdot \right) $ is $\psi-$Riemann-Liouville fractional integral where $0\leq \gamma <1$ \cite{Ze1}.

The paper is organized as follows: in section 2, we present some important preliminary definitions and result for the development: the $\psi-$Hilfer fractional derivative and the stability definitions of Hyers-Ulam-Rassias and Hyers-Ulam. In section 3, we present and discuss the main results of the paper, the stabilities of Hyers-Ulam-Rassias and Hyers-Ulam. %%%%%%%%%%%%%%%%%%%%%%%%%%%%%%%%%%%%%%%%%%%%%%%%%%%%%%%%%%%%%%%%%%%%%%%%%%%%%%%%%%%%%%%%%%%%%%%%%%%%%%%%%%%%%%%%%%%%%%%%%%%%%%%%%%%%%%%%%%%%%%%%%%%%%%%%%%%%%%%%%%%%%%%%%%%%%%%%%%%%%%%%%%%%%%%%%%%%%%%%%%%

\section{Preliminaries}
In this section, we present the $\psi-$Hilfer fractional derivative. In this sense, we also present the definition of stability of Hyers-Ulam-Rassias and Hyers-Ulam by means of the $\psi-$Hilfer fractional derivative and important results for the study of the fractional nonlinear differential equation Eq.(\ref{eq1}) and its stability.

Let $0<\alpha<1$, $I=[0,T]$ be a finite or infinite interval, $f$ an integrable function defined on $I$ and $\psi\in C^{1}(I)$ an increasing function such that $\psi ^{\prime }(t)\neq 0$, for all $t\in I$. The right-sided $\psi-$Hilfer fractional derivative, are defined by \cite{Ze1}
\begin{equation}  \label{eq3}
^{H}\mathbb{D}_{0+}^{\alpha ,\beta ;\psi }f\left( t\right) =I_{0+}^{\beta \left( 1-\alpha \right) ;\psi }\left( \frac{1}{\psi ^{\prime }\left(t\right)  }\frac{d}{dt}\right)I_{0+}^{\left( 1 -\beta \right) \left( 1-\alpha
\right) ;\psi }f\left( t\right).
\end{equation}

The following definition is an adaptation of Definition 1.1 as in \cite{sevgina}.

\begin{definition}
If for each continuously differentiable function $u\left( t\right) $ satisfying 
\begin{equation*}
\left\vert ^{H}\mathbb{D}_{a+}^{\alpha ,\beta ;\psi }u\left( t\right)
-f\left( t,u\left( t\right) \right) -\int_{0}^{t}k\left( t,s,u\left(
t\right) \right) ds\right\vert \leq \Phi \left( t\right) ,
\end{equation*}
where $\Phi \left( t\right) \geq 0$ for all $t,$ there exists a solution $u_{0}\left( t\right) $ of the Volterra integro-differential equation 
\textnormal{Eq.(\ref{eq1})} and a constant $C>0$ with 
\begin{equation*}
\left\vert u\left( t\right) -u_{0}\left( t\right) \right\vert \leq C\text{ }\Phi\left( t\right)
\end{equation*}
for all $t,$ where $C$ is independent of $u\left( t\right) $ and $u_{0}\left( t\right) $, then we say that the \textnormal{Eq.(\ref{eq1})} has the Hyers-Ulam-Rassias stability. If $\Phi \left( t\right) $ is
a constant function in the above inequalities, we say that \textnormal{Eq.(\ref{eq1})} has the Hyers-Ulam stability.
\end{definition}

The next Theorem \ref{teo1}, is very importance for the study of the main
purpose of this paper, the study of the stability of Hyers-Ulam and
Hyers-Ulam-Rassias.

\begin{theorem}\label{teo1} Let $\left( X,d\right) $ be a generalized complete metric space. Assume that $\Omega :X\rightarrow X$ a strictly contractive operator with Lipschitz constant $L<1.$ If there exists a nonnegative integer $k$ such that $d\left( \Omega ^{k+1}x,\Omega ^{k}x\right) <\infty $ for some $x\in X$, then following
\begin{enumerate}
\item The sequence $\left\{ \Omega ^{n}x\right\} $ converges to a fixed point $ x^{\ast }$ of $\Omega$;

\item $x^{\ast }$ is the unique fixed point of $\Omega $ in $X^{\ast }=\left\{y\in X/\text{ }d\left( \Omega ^{n}x,y\right) <\infty \right\}$;

\item If $y\in X^{\ast },$ then $d\left( y,X^{\ast }\right) \leq \dfrac{1}{1-L} d\left( \Omega y,y\right)$.
\end{enumerate}
\end{theorem}
\begin{proof}
See \cite{diaz}.
\end{proof}

%%%%%%%%%%%%%%%%%%%%%%%%%%%%%%%%%%%%%%%%%%%%%%%%%%%%%%%%%%%%%%%%%%%%%%%%%%%%%%%%%%%%%%%%%%%%%%%%%%%%%%%%%%%%%%%%%%%%%%%%%%%%%%%%%%%%%%%%%%%%%%%%%%%%%%%%%%%%%%%%%%%%%%%%%%%%%%%%%%%%%%%%%%%%%%%%%%%%%%%%%%%%

\section{Mains Results}
In this section, we introduce the Lipschitz condition, in order to present and discuss the main result of this paper, that is, the study of the stability of Hyers-Ulam-Rassias and Hyers-Ulam.

First, we introduce the following hypotheses.

(H1) Let $I=[0,T]$ be a given closed and bounded interval, with $T>0,$ and $ M,L_{f}$ and $L_{k}$ be positive constants with $0<ML_{f}+M^{2}L_{k}<1$. Suppose that $f:I\times \mathbb{R} \rightarrow \mathbb{R}$ is a continuous
function which satisfies a Lipschitz condition 
\begin{equation}  \label{eq5}
\left\vert f\left( t,u_{1}\right) -f\left( t,u_{2}\right) \right\vert \leq
L_{f}\left\vert u_{1}-u_{2}\right\vert, \text{ } \forall t\in I, \text{ } \forall u_{1},u_{2}\in \mathbb{R},
\end{equation}
$k:I\times I\times \mathbb{R} \rightarrow \mathbb{R}$ is a continuous function which satisfies a Lipschitz condition 
\begin{equation}  \label{eq6}
\left\vert k\left( t,s,u_{1}\right) -k\left( t,s,u_{2}\right) \right\vert
\leq L_{k}\left\vert u_{1}-u_{2}\right\vert, \text{ } \forall t,s\in I \text{ } \mbox{and}\text{ } \forall u_{1},u_{2}\in \mathbb{R}.
\end{equation}

\begin{theorem}\label{teo2} Assuming the hypotheses \textnormal{(H1)} and $\psi\in C[0,T]$ an increasing function such that $\psi'(t)\neq0$ on $I$. If a continuously differentiable function $u:I\rightarrow  \mathbb{R} $ satisfies
\begin{equation}\label{eq7}
\left\vert ^{H}\mathbb{D}_{0+}^{\alpha ,\beta ;\psi }u\left( t\right) -f\left( t,u\left( t\right) \right) -\int_{0}^{t}k\left( t,s,u\left( s\right) \right) ds\right\vert \leq \Phi \left( t\right) ,\text{ }\forall t\in I
\end{equation}
where $\Phi :I\rightarrow \left( 0,\infty \right)$ is a continuous function with
\begin{equation}\label{eq8}
I_{0+}^{\alpha ;\psi }\Phi \left( t\right) :=\frac{1}{\Gamma \left( \alpha  \right) }\int_{0}^{t}\psi ^{\prime }\left( \xi \right) \left( \psi \left( t\right) -\psi \left( \xi \right) \right) ^{\alpha -1}\Phi \left( \xi
\right) d\xi \leq M\Phi \left( t\right)
\end{equation}
for each $t\in I$, then there exists a unique continuous function $u_{0}:I\rightarrow \mathbb{R}$, such that
\begin{equation}\label{eq9}
u_{0}\left( t\right) =\frac{\left( \psi \left( t\right) -\psi \left(
0\right) \right) ^{\gamma -1}}{\Gamma \left( \gamma \right) }\sigma
+I_{0+}^{\alpha ;\psi }f\left( t,u_{0}\left( t\right) \right)
+I_{0+}^{\alpha ;\psi }\left[ \int_{0}^{\xi }k\left( t,s,u_{0}\left(
s\right) \right) ds\right] 
\end{equation}
with $I_{0+}^{1-\gamma ;\psi }u\left( 0\right) =\sigma ,$ $0<\alpha <1,$ $0\leq \beta \leq 1$ and
\begin{equation}\label{eq10}
\left\vert u\left( t\right) -u_{0}\left( t\right) \right\vert \leq \frac{M}{1-\left( ML_{f}+M^{2}L_{k}\right) }\Phi \left( t\right) ,\text{ }\forall t\in I.
\end{equation}
\end{theorem}

\begin{proof}

In fact, for $ v,w\in X,$ we set 
\begin{equation}  \label{eq11}
d\left( v,w\right) =\inf \left\{ c\in \left[ 0,\infty \right] /\text{ }\left\vert v\left( t\right) -w\left( t\right) \right\vert \leq C\Phi \left(t\right), \text{ }\forall t\in I\right\},
\end{equation}
where $X$ is the set of all real valued continuous functions on $I$.

Consider the operator $\Omega :X\rightarrow X$, defined by 
\begin{equation}
\Omega v\left( t\right) =\frac{\left( \psi \left( t\right) -\psi \left( 0\right) \right) ^{\gamma -1}}{\Gamma \left( \gamma \right) }\sigma +I_{0+}^{\alpha ;\psi }f\left( t,v_{0}\left( t\right) \right) +I_{0+}^{\alpha ;\psi }\left[ \int_{0}^{\xi }k\left( t,s,v_{0}\left( s\right) \right) ds\right]   \label{eq12}
\end{equation}%
for all $v\in I$.

1. $\Omega $ is strictly contractive on $X$.

Consider the $C_{vw}\in \left[ 0,\infty \right] $ be a constant with $d\left( v,w\right) \leq C_{vw}$ for any $v,w\in X,$ that is by Eq.(\ref{eq11}), we have 
\begin{equation}  \label{eq13}
\left\vert v\left( t\right) -w\left( t\right) \right\vert \leq C_{vw}\Phi \left( t\right) ,\text{ }\forall t\in I.
\end{equation}

So, by Eq.(\ref{eq5}), Eq.(\ref{eq6}). Eq.(\ref{eq8}), Eq.(\ref{eq12}) and Eq.(\ref{eq13}), we can write
\begin{eqnarray}
&&\left\vert \Omega v\left( t\right) -\Omega w\left( t\right) \right\vert  
\notag  \label{palmeiras} \\
&\leq &\frac{1}{\Gamma \left( \alpha \right) }\int_{0}^{t}\psi ^{\prime
}\left( \xi \right) \left( \psi \left( t\right) -\psi \left( \xi \right)
\right) ^{\alpha -1}\left\vert f\left( \xi ,v\left( \xi \right) \right)
-f\left( \xi ,w\left( \xi \right) \right) \right\vert d\xi +  \notag \\
&&+\frac{1}{\Gamma \left( \alpha \right) }\int_{0}^{t}\psi ^{\prime }\left(
\xi \right) \left( \psi \left( t\right) -\psi \left( \xi \right) \right)
^{\alpha -1}\int_{0}^{\xi }\left\vert k\left( t,s,v\left( s\right) \right)
-k\left( t,s,w\left( s\right) \right) \right\vert dsd\xi   \notag \\
&\leq &MC_{vw}L_{f}\Phi \left( t\right) +ML_{k}C_{vw}I_{0+}^{\alpha ;\psi }
\left[ \int_{0}^{\xi }\Phi \left( s\right) ds\right] .
\end{eqnarray}

Also, by inequality Eq.(\ref{palmeiras}), we obtain $\forall t\in I$ that is $d\left( \Omega v,\Omega w\right) \leq C_{vw}\Phi \left( t\right) \left[ ML_{f}+M^{2}L_{k}\right]$. 

Hence, we can conclude that $d\left( \Omega v,\Omega w\right) \leq \left[ ML_{f}+M^{2}L_{k}\right] d\left( v,w\right) $ for any $v,w\in X$, where $0<ML_{f}+M^{2}L_{k}<1$.

Thus, by Eq.(\ref{eq12}), there exists a constant $0<C<\infty $, with 
\begin{eqnarray*}
&&\left\vert \Omega w\left( t\right) -w_{0}\left( t\right) \right\vert  \\
&=&\left\vert \frac{\left( \psi \left( t\right) -\psi \left( 0\right)
\right) ^{\gamma -1}}{\Gamma \left( \gamma \right) }\sigma +I_{0+}^{\alpha
;\psi }f\left( t,w_{0}\left( t\right) \right) +I_{0+}^{\alpha ;\psi }\left[
\int_{0}^{\xi }f\left( t,s,w_{0}\left( s\right) \right) ds\right]
-w_{0}\left( t\right) \right\vert  \\
&\leq &C\Phi \left( t\right) 
\end{eqnarray*}
for arbitrary $w_{0}\in X$, $\forall t\in I$, since $f\left( \xi ,w_{0}\left( \xi \right) \right) $, $ k\left( t,s,w_{0}\left( s\right) \right) $ and $w_{0}\left( t\right) $ are bounded on their domains and $\underset{t\in I}{\min }\Phi \left(   t\right) >0$. Thus, the Eq.(\ref{eq11}) implies that $d\left( \Omega w_{0},w_{0}\right) <\infty $. Also, by Theorem \ref{teo1}.1, there exists a continuous function $u_{0}:I\rightarrow \mathbb{R} $ such that $\Omega ^{n}u_{0}\rightarrow u_{0}$ in $\left( X,d\right)$ and $\Omega u_{0}=u_{0}$, that is, $u_{0\text{ }}$ corresponds to the Eq.(\ref{eq9}) for all $t\in I$.

Since $w$ and $u_{0}$ are bounded on $I$ for any $w\in X$ and $\underset{ t\in I}{\min }\Phi \left( t\right) >0,$ there exists a constant $0<C_{vw}<\infty $ such that 
\begin{equation*}
\left\vert w_{0}\left( t\right) -w\left( t\right) \right\vert \leq C_{w}\Phi
\left( t\right)
\end{equation*}
for any $t\in I.$ We have $d\left( w_{0},w\right) <\infty $ for any $w\in X.$

Therefore, we have that $\left\{ w\in X/\text{ }d\left( w_{0},w\right) <\infty \right\} $ is equal to $X$. So, by Theorem \ref{teo1}.2, we conclude that $u_{0}$ given by Eq.(\ref{eq9}), is the unique continuous function.

From Eq.(\ref{eq7}) and by Theorem 5 \cite{Ze1}, we have 
\begin{eqnarray*}
&&\left\vert u\left( t\right) -\frac{\left( \psi \left( t\right) -\psi
\left( 0\right) \right) ^{\gamma -1}}{\Gamma \left( \gamma \right) }\sigma
-I_{0+}^{\alpha ;\psi }f\left( t,u\left( t\right) \right) -I_{0+}^{\alpha
;\psi }\left[ \int_{0}^{\xi }k\left( t,s,u\left( s\right) \right) ds\right]
\right\vert  \\
&\leq &\frac{1}{\Gamma \left( \alpha \right) }\int_{0}^{t}\psi ^{\prime
}\left( \xi \right) \left( \psi \left( t\right) -\psi \left( \xi \right)
\right) ^{\alpha -1}\Phi \left( \xi \right) d\xi .
\end{eqnarray*}

Then, by Eq.(\ref{eq8}) and Eq.(\ref{eq12}), we have 
\begin{equation*}
\left\vert u\left( t\right) -\Omega u\left( t\right) \right\vert \leq \frac{1%
}{\Gamma \left( \alpha \right) }\int_{0}^{t}\psi ^{\prime }\left( \xi
\right) \left( \psi \left( t\right) -\psi \left( \xi \right) \right)
^{\alpha -1}\Phi \left( \xi \right) d\xi \leq M\Phi \left( t\right) ,\text{ }%
\forall t\in I;
\end{equation*}
which implies
\begin{equation}  \label{eq16}
d\left( u,\Omega u\right) \leq M.
\end{equation}

Again by Theorem \ref{teo1}.3 and Eq.(\ref{eq16}), we conclude that 
\begin{equation*}
d\left( u,u_{0}\right) \leq \frac{1}{1-\left( ML_{f}+M^{2}L_{k}\right) }d\left( \Omega u,u\right) \leq \frac{M}{1-\left( ML_{f}+M^{2}L_{k}\right) },
\end{equation*}
which concludes the proof.

\end{proof}

Its important to note that, using the same hypotheses one can consider, always as a theorem, the stability of Hyers-Ulam-Rassias considering a limited and closed interval \cite{JVC}.

In what follows we introduce and prove a theorem associated with the stability involving Eq.(\ref{eq9}) a fractional Volterra integral-differential equation.

\begin{theorem} Let $0<\alpha<1$, $0\leq\beta\leq 1$ and $\psi\in C^{1}[0,T]$ an increasing function such that $\psi'(t)\neq 0$ for all $t\in I$. Also $L_{f}$ and $L_{k}$ be positive constants with $0<TL_{f}+\frac{{T}^{2}}{2}L_{k}<1$ and $I=\left[ 0,T\right] $ denote a given closed and bounded interval, with $T>0$. Suppose that $f:I\times \mathbb{R}\rightarrow \mathbb{R}$ is a continuous function which satisfies a Lipschitz condition \textnormal{Eq.(\ref{eq5})} and $k:I\times I\times \mathbb{R}\rightarrow \mathbb{R} $ is a continuous function which satisfies a Lipschitz condition \textnormal{Eq.(\ref{eq8})}. If for $\varepsilon \geq 0$ a continuously differentiable function $u:I\rightarrow \mathbb{R}$ satisfies
\begin{equation*}
^{H}\mathbb{D}_{0+}^{\alpha ,\beta ;\psi }u\left( t\right) -f\left( t,u\left( t\right) \right) -\int_{0}^{t}k\left( t,s,u\left( s\right) \right) ds\leq \varepsilon ,\text{ }\forall t\in I
\end{equation*}
then there exists a unique continuous function $u_{0}:I\rightarrow \mathbb{R}$ satisfying \textnormal{Eq.(\ref{eq9})} and
\begin{equation}\label{jose}
\left\vert u\left( t\right) -u_{0}\left( t\right) \right\vert \leq \frac{%
\left( \psi \left( T\right) -\psi \left( 0\right) \right) ^{\alpha
}\varepsilon }{\Gamma \left( \alpha +1\right) -\left( \psi \left( T\right)
-\psi \left( 0\right) \right) ^{\alpha }\left[ L_{f}+\frac{T}{2}L_{k}\right] 
}, \text{ } \forall t\in I.
\end{equation}
\end{theorem}

\begin{proof}
Initially, let $X$ the set of all real-valued continuous functions on $I.$ Furthermore, we define a generalized metric on $X$ by 
\begin{equation}\label{eq18}
d\left( v,w\right) =\inf \left\{ C\in \left[ 0,\infty \right] /\text{ }
\left\vert v\left( t\right) -w\left( t\right) \right\vert \leq C,\text{ }%
\forall t\in I\right\}.
\end{equation}

It is easy to see that $\left( X,d\right) $ is a complete generalized metric
space \cite{sevgina}.

Now, we define the operator $\Omega :X\rightarrow X$ by 
\begin{equation}\label{eq19}
\Omega v\left( t\right) =\frac{\left( \psi \left( t\right) -\psi \left(
0\right) \right) ^{\gamma -1}}{\Gamma \left( \gamma \right) }\sigma
+I_{0+}^{\alpha ;\psi }f\left( t,v\left( t\right) \right) +I_{0+}^{\alpha
;\psi }\left[ \int_{0}^{\xi }k\left( t,s,v\left( s\right) \right) ds\right] ,
\end{equation}
$\forall t\in I,$ for all $v\in I$.

We now prove that $\Omega $ is strictly contractive on the generalized
metric space $X.$ For any $v,w\in X,$ let $C_{vw}\in \left[ 0,\infty \right]$
be an arbitrary constant with $d\left( h,g\right) \leq C_{hg},$ that is, let
us suppose that 
\begin{equation}\label{eq20}
\left\vert v\left( t\right) -w\left( t\right) \right\vert \leq C_{vw},\text{ 
}\forall t\in I.
\end{equation}

Using the Eq.(\ref{eq5}), Eq.(\ref{eq6}), Eq.(\ref{eq19}) and Eq.(\ref{eq20}%
), we deduce 
\begin{eqnarray*}
&&\left\vert \Omega v\left( t\right) -\Omega w\left( t\right) \right\vert  \\
&\leq &\frac{L_{f}C_{vw}}{\Gamma \left( \alpha \right) }\int_{0}^{t}\psi
^{\prime }\left( \xi \right) \left( \psi \left( t\right) -\psi \left( \xi
\right) \right) ^{\alpha -1}d\xi +\frac{L_{k}C_{vw}}{\Gamma \left( \alpha
\right) }\int_{0}^{t}\psi ^{\prime }\left( \xi \right) \psi \left( t\right)
-\psi \left( \xi \right) z^{\alpha -1}\int_{0}^{\xi }d\xi  \\
&\leq &C_{vw}\left[ \frac{L_{f}\left( \psi \left( T\right) -\psi \left(
0\right) \right) ^{\alpha }}{\Gamma \left( \alpha +1\right) }+L_{k}\frac{T}{2%
}\frac{\left( \psi \left( T\right) -\psi \left( 0\right) \right) ^{\alpha }}{%
\Gamma \left( \alpha +1\right) }\right] ,\text{ }\forall t\in I.
\end{eqnarray*}

We conclude that 
\begin{equation*}
d\left( \Omega v,\Omega w\right) \leq \left( L_{f}\frac{\left( \psi \left(
T\right) -\psi \left( 0\right) \right) ^{\alpha }}{\Gamma \left( \alpha
+1\right) }+L_{k}\frac{T}{2}\frac{L_{f}\left( \psi \left( T\right) -\psi
\left( 0\right) \right) ^{\alpha }}{\Gamma \left( \alpha +1\right) }\right)
d\left( v,w\right),\text{ } \forall v,w\in X.
\end{equation*}

Let $w_{0}$ any arbitrary element in $X.$ Then there exists a constant $0<C<\infty $ 
\begin{eqnarray*}
&&\left\vert \Omega w_{0}\left( t\right) -w_{0}\left( t\right) \right\vert  
\notag \\
&=&\left\vert \dfrac{\left( \psi \left( t\right) -\psi \left( 0\right)
\right) ^{\gamma -1}}{\Gamma \left( \gamma \right) }\sigma +I_{0+}^{\alpha
;\psi }f\left( t,w_{0}\left( t\right) \right) +I_{0+}^{\alpha ;\psi }\left[
\int_{0}^{\xi }k\left( t,s,w_{0}\left( s\right) \right) ds\right]
-w_{0}\left( t\right) \right\vert \leq C,
\end{eqnarray*}
$\forall t\in I$.

Since $\ f\left( \xi ,w_{0}\left( \xi \right) \right) ,$ $k\left( t,s,w_{0}\left( s\right) \right) $ and $w_{0}\left( t\right) $ are bounded on their domain, Eq.(\ref{eq18}) implies that $d\left( \Omega
w_{0},w_{0}\right) <\infty .$

Therefore, according to Theorem \ref{teo1}.1, there exists a continuous
function $u_{0}:I\rightarrow \mathbb{R} $ such that $\Omega
^{n}w_{0}\rightarrow u_{0}$ in $\left( X,d\right) $ as $n\rightarrow \infty $
and such that $\Omega u_{0}=u_{0},$ that is, $u_{0}$ satisfies Eq.(%
\ref{eq9}) for every $t\in I$. As in the proof of Theorem \ref{teo2} it can
be verify that $\left\{ w\in X/\text{ }d\left( w_{0},w\right) <\infty
\right\} =X.$ Due to Theorem \ref{teo1}.2 $u_{0}$ given by Eq.(\ref{eq9}),
is the unique continuous function.

From Eq.(\ref{eq7}) and by Theorem 5 \cite{Ze1}, we have
\begin{eqnarray*}
&&\left\vert u\left( t\right) -\frac{\left( \psi \left( t\right) -\psi
\left( 0\right) \right) ^{\gamma -1}}{\Gamma \left( \gamma \right) }\sigma
-I_{0+}^{\alpha ;\psi }f\left( t,u_{0}\left( t\right) \right)
-I_{0+}^{\alpha ;\psi }\left[ \int_{0}^{\xi }k\left( t,s,u_{0}\left(
s\right) \right) ds\right] \right\vert  \\
&\leq &\varepsilon \frac{\left( \psi \left( T\right) -\psi \left( 0\right)
\right) ^{\alpha }}{\Gamma \left( \alpha +1\right) },
\end{eqnarray*}
for all $t\in I$, which implies
\begin{equation*}
d\left( u,\Omega u\right) \leq \varepsilon \frac{\left( \psi \left( T\right)
-\psi \left( 0\right) \right) ^{\alpha }}{\Gamma \left( \alpha +1\right) }.
\end{equation*}

Lastly, Theorem \ref{teo1}.3 together with Eq.(\ref{eq11}) implies that 
\begin{equation*}
\left\vert u\left( t\right) -u_{0}\left( t\right) \right\vert \leq \frac{%
\left( \psi \left( T\right) -\psi \left( 0\right) \right) ^{\alpha
}\varepsilon }{\Gamma \left( \alpha +1\right) -\left( \psi \left( T\right)
-\psi \left( 0\right) \right) ^{\alpha }\left[ L_{f}+\frac{T}{2}L_{k}\right] 
}
\end{equation*}
that is the inequality Eq.(\ref{jose}) be true for all $\ t\in I$.
\end{proof}

\bibliography{ref}

\begin{thebibliography}{10}

\bibitem{Ze}
J.~Vanterler~da C.~Sousa, E.~{Capelas}~de Oliveira,  L. A. Magna, Fractional Calculus and the {ESR} test, AIMS Math., 2(4) (2017) 692-705.

\bibitem{RHM}
R. Herrmann, Fractional calculus: An {I}ntroduction for {P}hysicists, {W}orld {S}cientific {P}ublishing {C}ompany, {S}ingapore, (2011).

\bibitem{wang}
J. Wang, Y. Zhou, Mittag-{L}effler-{U}lam stabilities of fractional evolution equations, Appl. Math. Lett., 25(4) (2012) 723-728.

\bibitem{sousa}
J. Vanterler da C. Sousa, E. Capelas Oliveira, On the {U}lam-{H}yers-{R}assias stability for nonlinear fractional differential equations using the $\psi$-{H}ilfer operator,  submitted, (2017).

\bibitem{JinRong}
J. Wang, L. Lv, Y. Zhou, New concepts and results in stability of fractional differential equations, Commun. Nonlinear Sci. Numer. Simulat., 17(6) (2012) 2530-2538.

\bibitem{wei}
W. Wei, X. Li, Xia Li, New stability results for fractional integral equation, Comp. and Math. with Appl., 64(10) (2012) 3468-3476.

\bibitem{wang1}
Rong-Nian Wang, J. Liu De-Han Chen, Abstract fractional integro-differential equations involving nonlocal initial conditions in $\alpha$-norm, Adv. Diff. Equa., 2011(1) (2011) 25.


\bibitem{muniya}
P. Muniyappan, Stability of a class of fractional integro-differential equation with nonlocal initial condition, Acta Math. Universitatis Comenianae, (2017) 1-12.


\bibitem{Ze1}
J.~Vanterler~da C.~Sousa, E.~{Capelas}~de Oliveira, On the $\psi$-{H}ilfer fractional derivative, submitted, (2017).

\bibitem{sevgina}
S. Sevgina, H. Sevlib, Stability of a nonlinear {V}olterra integro-differential equation via a fixed point approach, J. Nonlinear Sci. Appl., 1(9) (2016) 200-207.



\bibitem{diaz}
J. B. Diaz, B. Margolis, A fixed point theorem of the alternative, for contractions on a generalized complete metric space, Bull. Amer. Math. Soc., 74(2) (1968) 305-309.


\bibitem{JVC} J. Vanterler da C. Sousa, Erythrocyte sedimentation rate: {A} fractional model (In Portuguese), {P}hD {T}hesis, {I}mecc-{U}nicamp, {C}ampinas, (2017)



\end{thebibliography}
\bibliographystyle{plain}

\end{document}